\documentclass[10pt,a4paper]{article}
\usepackage[utf8]{inputenc}
\usepackage{amsmath}
\usepackage{amsfonts}
\usepackage{amssymb}
\usepackage{graphicx}
\usepackage{float}
\usepackage{dsfont}
\usepackage[normalem]{ulem}

\usepackage{amsthm}
\usepackage{bbm}
\usepackage{xcolor}
\usepackage{mathtools}

\newtheorem{theorem}{Theorem}[section]
\newtheorem{lemma}[theorem]{Lemma}
\newtheorem{proposition}[theorem]{Proposition}
\newtheorem{corollary}[theorem]{Corollary}

\theoremstyle{definition}

\newtheorem{definition}[theorem]{Definition}
\newtheorem{remark}[theorem]{Remark}

\newcommand{\R}{{\mathbb{R}}}
\newcommand{\Z}{{\mathbb{Z}}}
\newcommand{\C}{{\mathbb{C}}}
\newcommand{\CP}{{\mathbb{CP}}}
\newcommand{\RP}{{\mathbb{RP}}}
\newcommand{\bS}{\mathbb{S}}
\newcommand{\ehzcap}{c_{_{\rm EHZ}}}

\newcommand{\gw}{\underline{c}}

\def\CC{\mathbb{C}}
\def\RR{\mathbb{R}}

\newcommand{\deta}{\dot{\eta}}

\newcommand{\bigzero}{\mbox{\normalfont\Large\bfseries 0}}

\def\eps{{\varepsilon}}

\begin{document}

\title {Quantitative Results on  Symplectic Barriers}
	\author{Pazit Haim-Kislev, Richard Hind, Yaron Ostrover}
	\maketitle
	\begin{abstract}
In this paper we present some quantitative results concerning symplectic barriers. In particular, we answer a question raised by Sackel, Song, Varolgunes, and Zhu regarding the symplectic size of the $2n$-dimensional Euclidean ball with a codimension-two linear subspace removed.    
\end{abstract}

\section{Introduction and Results} \label{sec-introduction}

In a recent work~\cite{HKHO}, we established a new type of rigidity for symplectic embeddings that originates from obligatory intersections with symplectic submanifolds.
Inspired by the terminology introduced by Biran~\cite{B} for analogous results regarding Lagrangian submanifolds, we refer to such symplectic obstructions as symplectic barriers.
More precisely, if $(M, \omega)$ is a $2n$-dimensional symplectic manifold, a symplectic submanifold $S \subset M$ is said to be a symplectic barrier if
$$ c(M \setminus S, \omega) <  c(M,\omega),$$
where $c$ is some (normalized) symplectic capacity. Recall that symplectic capacities are numerical invariants that, roughly speaking, measure the size of a symplectic manifold (see, e.g., Chapter 2 of~\cite{HZ1}). More precisely, 
\begin{definition} \label{def-capacity}
	A  symplectic capacity is a map which associates to every symplectic manifold $(M,\omega)$ an element of $[0, \infty]$ with the following properties:
	\begin{itemize}
	\item  $c(M,\omega) \leq c(N, \tau)$ if $(M,\omega)  \overset{\mathrm s}{\hookrightarrow } (N,\tau)$ (Monotonicity)
	\item  $c(M,\alpha \omega) = |\alpha| c(M,\omega)$ for all $\alpha \in \R$, $\alpha \neq 0$ (Conformality)
	\item $c(B^{2n}(1)) = 1 = c(Z^{2n}(1)) $ (Nontriviality and Normalization)
	\end{itemize}
\end{definition}
\noindent Here $B^{2n}(r)$ denotes the $2n$-dimensional Euclidean ball with radius $\sqrt{r/\pi}$, $Z^{2n}(r)$  denotes the cylinder $B^2(r) \times \R^{2n-2}$,
and $\overset{\mathrm s}{\hookrightarrow }$ stands for a symplectic embedding. Note that both the ball $B^{2n}(r)$ and the cylinder $Z^{2n}(r)$ are equipped with the standard symplectic form $\omega = dx \wedge dy$ on ${\mathbb R}^{2n}$. 
Two examples of  symplectic capacities, which naturally arise from Gromov's celebrated non-squeezing theorem~\cite{Gr},
are the Gromov width and the cylindrical capacity:
\begin{align*}
\gw(M) = \sup \{ r : B^{2n}(r) \xhookrightarrow{\mathrm s} M \}, \ \
\overline{c}(M) = \inf \{r : M \xhookrightarrow{s} Z^{2n}(r) \}.
\end{align*}
Another important example is the Hofer-Zhender capacity $c_{_{\rm HZ}}$, which is closely related with Hamiltonian dynamics (see, e.g.,~\cite{HZ}). 
It follows immediately from Definition~\ref{def-capacity} that  $\gw(M,\omega) \leq c(M,\omega) \leq \overline{c}(M,\omega)$ for any symplectic capacity $c$.
For more information on symplectic capacities, see e.g., the survey~\cite{CHLS}. 

\medskip

In this paper we present some quantitative results concerning symplectic barriers. Our first result answers a question raised in Section 6.2 of~\cite{SSVZ} regarding the capacity of the $2n$-dimensional Euclidean ball in ${\mathbb R}^{2n}$ with a codimension-two linear subspace removed, where the classical K\"ahler angle is used to measure the ``defect" of the subspace from being complex. More precisely, let $E_t \subset {\mathbb R}^{2n}$ be a codimension-two linear subspace with  K\"ahler angle $t$, i.e, the unit outer normals $n_1 \perp n_2$ to $E_t$ satisfy $|\omega(n_1,n_2)| = t$.  The following result  
shows that the $(2n-2)$-ball $E_t \cap B^{2n}(1)$ is a symplectic barrier in $B^{2n}(1)$  when $0 \leq t <1$.

\begin{theorem} \label{thm-complement-plane} 
Let $n>1$ and $0 \leq t \leq 1$.
For any symplectic capacity $c$ 
one has
$$c(B^{2n}(1) \setminus E_t) = {\dfrac {1+t}{2}}.$$
\end{theorem}
We remark that the complex case $t=1$  follows from 
Proposition 1.6 in~\cite{HKHO} (cf. Theorem 3.1.A in~\cite{MP}). For the Gromov width, the Lagrangian case $t=0$ follows from~\cite{B}, where it is proved that $\gw( \CP^n \setminus \RP^n) = {\frac 1 2}$. Moreover, one can check that Theorem 1.3 in~\cite{SSVZ} implies that $\gw (B^{2n}(1) \setminus E_0) = {\overline c}(B^{2n}(1) \setminus E_0) = {\frac 1 2}$.  

\medskip

Our next result concerns the symplectic barriers introduced in~\cite{HKHO}.
For $\varepsilon >0$ denote by $\Sigma_\eps$ the following union of symplectic codimension-two subspaces in ${\mathbb R}^{2n} \simeq \C^n$:
\begin{align*}
\Sigma_\eps := \bigcup \{(z_1,z_2,\ldots,z_n) \in \C^n : z_n \in \eps \Z^{2} \}.
\end{align*}
Moreover, define $\Sigma_\eps^t$ to be a linear image of $\Sigma_\eps$ such that the K\"ahler angle of the corresponding planes is $t$, i.e., $|\omega(n_1,n_2)| = t$, where   $n_1 \perp n_2$ are the unit outer normals to the subspaces in $\Sigma^t_{\eps}$. 
Note that any two such configurations of subspaces with the same K\"ahler angle  are unitarily equivalent.\
We also note that when $\varepsilon$ is sufficiently large, 
the intersection $B^{2n}(1) \cap \Sigma_\eps^t$ becomes $B^{2n}(1) \cap E_t$, where $E_t$ is a single codimension-two subspace of K\"ahler angle $t$ as above. Thus we are especially interested in the case when $\varepsilon$ is small. In~\cite{HKHO} it was proved that for small $\varepsilon$, the configurations $\Sigma_\eps^t$ are symplectic barriers of the ball $B^{2n}(1)$ with respect to any (normalized) symplectic capacity. Here we provide more precise bounds   for the symplectic size of the complement of $\Sigma^t_{\eps}$ in $B^{2n}(1)$  when $\varepsilon \rightarrow 0$. 

\begin{theorem} \label{Thm-union-of-planes}
    For any $t \in (0,1)$ and $n>1$
    $$ 
    \lim_{\eps \to 0} c_{_{\rm HZ}}(B^{2n}(1) \setminus \Sigma_\eps^t)
    =
    \lim_{\eps \to 0} \overline{c}(B^{2n}(1) \setminus \Sigma_\eps^t) = t. $$
\end{theorem}
We suspect that Theorem~\ref{Thm-union-of-planes} also holds for the Gromov width. This is supported by the following claim that provides an almost exact lower bound:
\begin{theorem} \label{thm-GW-capacity-Gromov-width} For any $t \in (0,1)$, any $\varepsilon > 0$, and $n>1$ 
     $$f(t) \leq \underline{c}(B^{2n}(1) \setminus \Sigma_\eps^t),$$
     where $f(t)$ is an explicit function given by \eqref{the-lower-bound-function} below and satisfies $f(t) \geq t-0.07$.
\end{theorem}
While Theorem \ref{thm-complement-plane} shows that the complement of a single codimension-two linear subspace with K\"ahler angle $t$ has capacity $\frac{1+t}{2}$, the  two theorems above show that the symplectic size of 
the complement of a large number of such spaces 
is strictly smaller, and takes a value around the K\"ahler angle $t$.

\medskip 

\noindent {\bf Acknowledgements:} 
We are  grateful to Yael Karshon for numerous enlightening discussions, particularly for her generous sharing of ideas concerning the proof of Proposition~\ref{lem-lower-bound-intersection}, which is crucial for the results in Section 3. 
P. H-K. and Y.O.  were partially supported by the ISF grant No.~938/22, and R.H. by the Simons Foundation grant No. 663715.

\section{The Complement of a Single Subspace} \label{sec-removing-2-codim-hyperplane}
In this section we prove Theorem~\ref{thm-complement-plane}. We first introduce  the following notations. 
Equip  ${\mathbb C}^n \simeq {\mathbb R}^{2n}$ with coordinates 
$(z_1,\ldots,z_n)$, where $z_j = x_j +i y_j$, 
and with the standard symplectic form $\omega = dx \wedge dy$. Let $P_t$ be the real two-dimensional plane in $\CC^2$ spanned by the two vectors $(s,t)$ and $(0,i)$, where $t,s$ are positive real numbers satisfying $t^2 + s^2 = 1 $. It is not hard to check that using a  linear unitary transformation in ${\mathbb C}^n$, we can assume without loss of generality that  any codimension-two linear space $E_t \subset {\mathbb R}^{2n}$ with  
 K\"ahler angle $t$ is of the form $P_t \times {\mathbb C}^{n-2}$. This implies that the proof of Theorem~\ref{thm-complement-plane} is, roughly speaking, four-dimensional. 

\medskip
The strategy for the proof of Theorem~\ref{thm-complement-plane} is as follows: first we prove the required lower bound for the Gromov width by an explicit embedding of a 4-dimensional ball in $B^4(1) \setminus P_t$ (Proposition~\ref{c3}), and then  extend the argument to any dimension in Proposition~\ref{c3n}. Next, in Proposition~\ref{thm-embedding-to-ball-with-Lagrangian-removed}
we develop the main ingredient needed for the required upper bound for the cylindrical capacity, which in turn is proved in Proposition~\ref{prop-cylindrical-upper-bnd}.

 \medskip 
We start with some preparation. First, note that the plane $P_t$ lies in the hyperplane $\Sigma := \{ y_1=0 \}$. Moreover, let ${\rm proj}_{z_k} : \CC^2 \to \CC$ be the  projection onto the $z_k$-plane, for $k=1,2$. Then, one has $${\rm proj}_{z_1}(P_t \cap B^4(1)) = \{  x_1 \in [-s,s], \, y_1=0\} \ {\rm and}  \ {\rm proj}_{z_2}(P_t \cap B^4(1))=E,$$ where $E$ is the ellipse with axes $ [-t/\sqrt{\pi}, t/\sqrt{\pi}] \times \{0\}$, $\{0\} \times [-1/\sqrt{\pi}, 1/\sqrt{\pi}]$, and area $t$ (see Figure~\ref{figure-projections}).

\begin{figure}[ht]
\centering
	\includegraphics[width=0.7\textwidth]{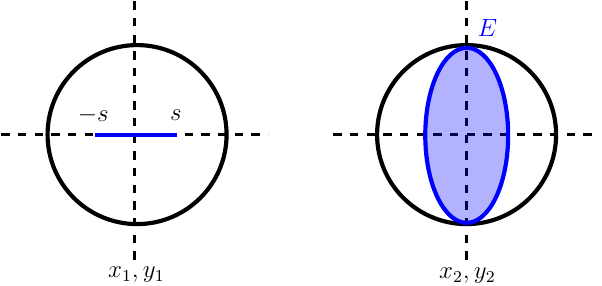}
 \caption{The projections of $P_t \cap B^4(1)$.}
 \label{figure-projections}
 \end{figure}

\vskip 5pt

Set $E^+ = E \cap \{x_2 \ge 0\}$, $E^- = E \cap \{x_2 \le 0\}$, $\partial E^+ = \partial E \cap \{x_2 \ge 0\}$, and  $\partial E^- = \partial E \cap \{x_2 \le 0\}$. 
Note that any subset $U$ of the intersection $\Sigma \cap B^4(1)$ satisfying   ${\rm proj}_{z_2}(U) \cap \partial E^+ = \emptyset$ can be displaced from $P_t$ using a Hamiltonian diffeomorphism of $\Sigma \cap B^4(1)$ that sends points of the form $(x_1,0,x_2, y_2)$ to $(f(x_1, x_2, y_2), 0, x_2, y_2)$, with $f(x_1, x_2, y_2) \ge x_1$. This can be done, e.g., via  the following simple lemma.

\begin{lemma} \label{lem-pushing-along-vector-field} 
Let $N \subset (M, \omega)$ be a submanifold
 and let $X_N$ be a vector field on $N$ tangent to $\ker(\omega|_N)$ whose time-$1$ flow defines a diffeomorphism $\psi_f$ of $N$.
Then there exists a Hamiltonian diffeomorphism of $M$ which preserves $N$ and restricts to $\psi_f$ on $N$.
\end{lemma}

\begin{proof}
As the $1$-form $\eta = X_N \rfloor \omega$ vanishes on $TN$, one can find a function $H : M \rightarrow {\mathbb R}$ which vanishes on $N$, and satisfies $dH = \eta$. Clearly the function $H$ generates the required  Hamiltonian diffeomorphism.
\end{proof}

With these preliminaries in place, we turn now to prove the required lower bound for the Gromov width in dimension four.

\begin{proposition}\label{c3} 
 For any $t \in (0,1)$ one has,
$$\underline{c}(B^4(1) \setminus P_t) \ge \frac{1+t}{2}.$$
\end{proposition}

\begin{proof}
By the remark proceeding Lemma~\ref{lem-pushing-along-vector-field}, for the proof of the proposition it suffices to find a  symplectic embedding $ \phi : B^4({\frac {1+t} 2}) \overset{\mathrm s}{\hookrightarrow } B^4(1)$ with image ${\bf B}$ satisfying ${\rm proj}_{z_2}({\bf B} \cap \Sigma) \cap \partial E^+ = \emptyset$. Denote by $D(r)$ the disk of area $r$ centered at the origin. 
Note that $\partial E^+$ divides the disc $D(1)$ into two regions, and set $R$ to be the one with area $\frac{1+t}{2}$. Next, let $S_h := R \cap D(1-h)$ (see Figure \ref{sh-def-fig}). Note that $S_h$ is the set of all points in $ R$ such that the fibers of the map ${\rm proj}_{z_2}$ in $B^4(1)$ have area at least $h$. It is not hard to check that the area of $S_h$ is at least $\frac{1-h}{2} + \frac{t}{2}\sqrt{1-h}$. Indeed,  the area of $S_h \cap \{x_2 \leq 0\}$ is ${\frac {1-h} 2}$, while the area of $S_h \cap \{x_2 \geq 0\}$ can be bounded from below by half the area of an ellipse centered at the origin with radii $r_1$ and $r_2$, where $\pi r_1^2 = 1-h$ and $\pi r_2^2 = t^2$. 

\vskip 5pt

We shall construct an area preserving map $\varphi$ from $D({\frac {1+t} {2}})$ to $R$, such that the product ${\rm Id} \times \varphi$ is the required symplectic embedding $\phi$. Note that the image ${\bf B} \subset B^4(1)$ if $\varphi (D({\frac {1+t} {2}} -h)) \subset S_h$ for all $0 \le h \le \frac{1+t}{2}$. Indeed, let $(z_1',z_2') \in B^4 ({\frac {1+t} 2})$ such that $z_1' \in D(h')$ and $z_2' \in D({\frac {1+t} 2} -h')$ for some $h'$. 
If $\varphi (z_2') \in S_{h'}$, then, as $S_{h'} \subset D(1-h')$, one as $(z_1',\varphi(z_2')) \subset B^4(1)$. 
Using Lemma 3.1.5 in~\cite{schl} (see Figure \ref{figure-varphi}), one can construct such a map $\varphi$ since $${\rm Area} \, \bigl (D({\frac {1+t} {2}} -h) \bigr ) = \frac{1+t}{2} - h \le \frac{1-h}{2} + \frac t2 \sqrt{1-h} \le {\rm Area} \, (S_h),$$ for all $0 \leq h \leq {\frac {1+t} {2} }$. 
Finally, as the $z_2$-component of $\phi$ lies in $R$, such a map automatically satisfies ${\rm proj}_{z_2}({\bf B} \cap \Sigma) \cap \partial E^+ = \emptyset$, as required.
This completes the proof of the proposition. 
\end{proof}

\begin{figure}[h]
\centering
\includegraphics[width=0.4\textwidth]{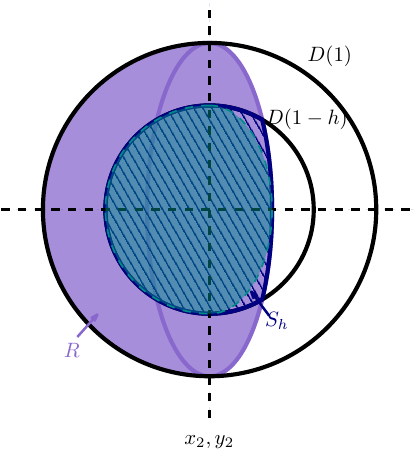}
 \caption{The domain $S_h$ in blue, and the domain $R$ in purple. The domain in green is used to bound the area of $S_h$ from below.}
 \label{sh-def-fig}
 \end{figure}

We can extend Proposition \ref{c3} to higher dimensions as follows.

\begin{proposition}\label{c3n} 
 For any $t \in (0,1)$ and $n \ge 2$
$$\underline{c}(B^{2n}(1) \setminus (P_t \times \CC^{n-2})) \ge \frac{1+t}{2}.$$
\end{proposition}

\begin{proof}
We would like to construct a Hamiltonian function $G(z_1, \dots, z_n, t)$ whose corresponding time-$1$ flow maps the ball $B^{2n}({\frac {1+t} 2})$ into $B^{2n}(1) \setminus (P_t \times \CC^{n-2}))$.
From the proof of Proposition \ref{c3} it follows that there exists a Hamiltonian function $H(z_1, z_2, t)$ whose time-$1$ flow maps the ball $B^4({\frac {1+t} 2})$ into $B^4(1) \setminus P_t$. It follows that for every $r>0$ the time-1 map of the Hamiltonian function  $H_r(z_1, z_2, t) := rH(z_1 / \sqrt{r}, z_2 / \sqrt{r}, t)$ maps $B^4({\frac {r(1+t)} 2})$ into $B^4(r) \setminus P_t$.
Next, for a point $z=(z_1,\ldots,z_n) \in B^{2n}({\frac {1+t} 2})$, set $\xi(z) = \pi \sum_{k=3}^n |z_k|^2$. Note that $(z_1,z_2) \in B^4({\frac {1+t} 2} - \xi(z)) \subset B^4 \left ((1-\xi(z)) {\frac {1+t} 2} \right)$. Hence, since  $\phi_{H_{1-\xi(z)}}(z_1,z_2) \in B^4(1-\xi(z)) \setminus P_t$, the time-$1$ map of the Hamiltonian function $$G(z_1, \dots, z_n, t) := H_{1 - \xi(z)}(z_1, z_2, t),$$  satisfies 
$$ \phi_G (z_1,\ldots z_n) \in B^{2n}(1) \setminus (P_t \times \CC^{n-2})) $$ as required (cf. \cite{bh}, Section 2.1). This completes the proof of the proposition. 
\end{proof}

\begin{figure}[h]
\centering
\includegraphics[width=0.6\textwidth]{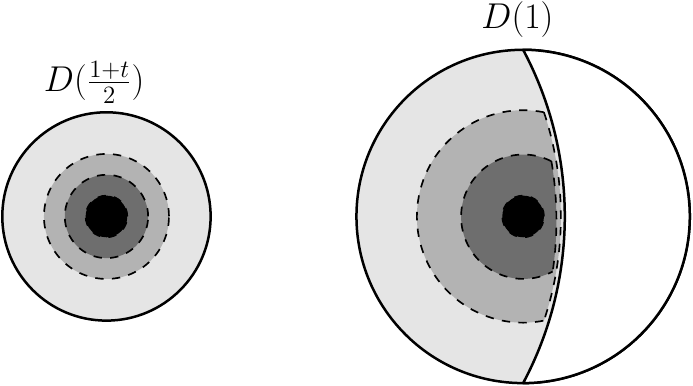}
 \caption{The map $\varphi$ from $D(\frac{1+t}{2})$ into $R$.}
 \label{figure-varphi}
 \end{figure}

Next we turn to establish the upper bound in Theorem~\ref{thm-complement-plane}. For this, let $L$ be the Lagrangian plane spanned by $(0,i)$ and $(1,0)$. 
Note that both $L$ and $P_t$ lie in the hyperplane $\Sigma := \{ y_1=0 \}$, and one has ${\rm proj}_{z_2}(L) = \{x_2=0\}$.
The main ingredient we need is the following:

\begin{theorem} \label{thm-embedding-to-ball-with-Lagrangian-removed} 
Let ${\bf K} \subset B^4(1) \setminus P_t$ be a compact subset. 
Then there exists a symplectomorphism  $\phi$ of $\CC^2$ with $\phi({\bf K}) \subset B^4(1+t) \setminus L$.
\end{theorem}

\begin{proof}

Set  $I := {\bf K} \cap \Sigma = A \sqcup B$ with $A$, $B$ defined as follows:

$$A := \{ p \in I \, | \, p + (\lambda,0) \in P_t \ \mathrm{for \, some} \, \lambda<0 \},$$
$$B := \{ p \in I \, | \, p + (\lambda,0) \in P_t \ \mathrm{for \, some} \, \lambda>0 \}.$$
Since $P_t$ is a graph over the $z_2$-plane, $I$ is the disjoint union of $A$ and $B$.
Further, since $s,t$ are positive we have that 
$\{ (z_1, z_2) \in P_t \, | \, z_2 \in E^+\} \subset \{x_1 \ge 0\}$ and $ \{ (z_1, z_2) \in P_t \, |  \, z_2 \in E^-\} \subset \{x_1 \le 0\}.$
This implies in particular that 
${\rm proj}_{z_2}(A) \subset \{x_2 \le 0\} \cup E^+$ and  ${\rm proj}_{z_2}(B) \subset \{x_2 \ge 0\} \cup E^-$. 
Indeed, for $p \in A$, if $x_2(p) \geq 0$, then $x_1(p) \geq 0$, and then $p + (\lambda,0) \in B^4(1) \cap P_t$. Hence ${\rm proj}_{z_2}(p) = {\rm proj}_{z_2}(p + (\lambda,0)) \in E^{+}$. A similar argument holds for  points in $B$.

\vskip 5pt Our proof has two steps. In Step 1 we apply a symplectic diffeomorphism to ${\bf K}$, with support in $B^4(1+t)$, moving first
the  subsets $A$ and $B$ away from $x_1=0$, and then moving ${\bf K}$ sufficiently away from the $z_2$-axis.
In Step 2 we describe a Hamiltonian diffeomorphism of $B^4(1+t)$ displacing the re-positioned ${\bf K}$ obtained in Step 1 from $L$ as required.

\vskip 5pt {\it Step 1.} The repositioning of ${\bf K}$ is achieved via the following two lemmas.

\begin{lemma}\label{one}
For every $\delta >0$ sufficiently small,
there exists a Hamiltonian diffeomorphism $\psi_1$ with compact support in $B^4(1) \setminus P_t$ 
such that the sets $A$ and $B$ defined for $\psi_1({\bf K})$ satisfy $A \subset \{x_1 > \delta\}$ and $B \subset \{x_1 < -\delta\}$.
\end{lemma}

\begin{proof}
We can find such a diffeomorphism $\psi_1$ which preserves $\Sigma$ by applying 
Lemma \ref{lem-pushing-along-vector-field},
since moving points of $A$ in the positive $x_1$-direction, and points of $B$ in the negative $x_1$-direction does not introduce intersections with $P_t$.
\end{proof}

To simplify notations, in what follows we denote   the image $\psi_1({\bf K})$ provided by Lemma~\ref{one} also by ${\bf K}$.

\vskip 5pt 

\begin{lemma}\label{two} 
For every $\delta >0$ sufficiently small there is a symplectic diffeomorphism 
{\color{black} $\psi_2 : {\mathbb R}^4 \rightarrow {\mathbb R}^4$ which is the identity on $\Sigma$, and satisfies }
 $$\psi_2({\bf K}) \subset \bigl ( \{ \pi|z_1|^2 > t \} \cup  \{|x_1| > \delta \ {\rm and} \  |y_1| < \delta \} \bigr ) \cap B^4(1+t).$$
\end{lemma}

\begin{proof} We use symplectic polar coordinates on the $z_1$-plane, with $R = \pi |z_1|^2$ and $\theta \in S^1 = \R / \Z$.
Let $f: S^1 \to [0,t]$ satisfy $f(\theta)=t$ when both $| \theta |> \delta$ and $| \theta - 1/2| > \delta$, and  $f=0$ when $\theta =0$ or $\theta = 1/2$. Then, consider the closed $1$-form $\eta = f(\theta) d \theta$. As ${\bf K}$ is disjoint from $\{z_1=0\}$, the form $\eta$ defines a symplectic isotopy of ${\bf K}$ increasing the $R$ coordinate by $f(\theta)$ and preserving $z_2$. As $0 \le f(\theta) < t$, the ball remains disjoint from $\{ z_1=0 \}$ but stays within $B^4(1+t)$, as required (see Figure \ref{figure-psi-2}). This completes the proof of the lemma. 
\end{proof}

\begin{figure}[h]
\centering
\includegraphics[width=0.6\textwidth]{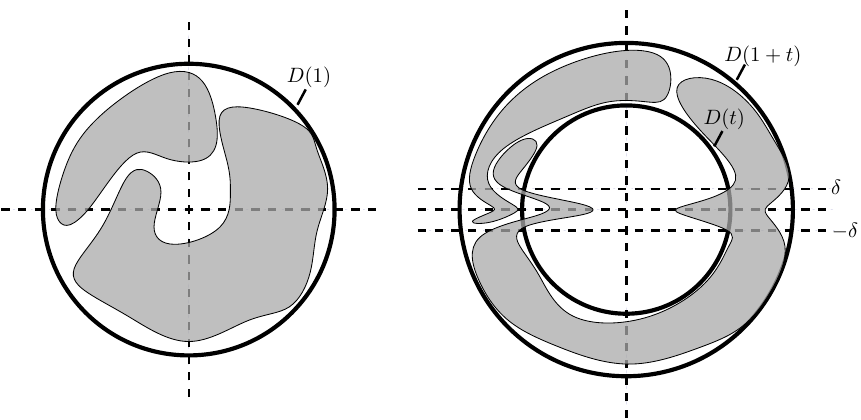}
 \caption{An illustration of the area preserving map $\psi_2$ which ``pushes" ${\bf K}$ into the set $\{ \pi|z_1|^2 > t \} \cup  \{|x_1| > \delta \ {\rm and} \  |y_1| < \delta \}$.}
 \label{figure-psi-2}
 \end{figure}

\vskip 5pt {\it  Step 2.} Displacing the repositioned ${\bf K}$ from $L$.

\vskip 5pt
To further simplify notations, using Step 1, in what follows we assume that
%
${\bf K}  \subset \bigl( \{ \pi|z_1|^2 > t \} \cup  \{|x_1| > \delta \ {\rm and} \  |y_1| < \delta \} \bigr ) \cap B^4(1+t)$ for some sufficiently small $\delta >0$. 
 In addition,  note that ${\bf K} \cap \Sigma = A \sqcup B \subset B^4(t)$, where $A \subset \{x_1 > \delta\}$ and $B \subset \{x_1 < -\delta\}$. 
Our goal is to find a Hamiltonian diffeomorphism $\phi$ of $B^4(1+t)$ which displaces ${\bf K}$ from $L$. 
We must consider both $I = {\bf K} \cap \Sigma$ and $J = {\bf K} \cap \Sigma^c$, and give sufficient conditions for a Hamiltonian function to generate a diffeomorphism displacing $I$ from $L$, while leaving $J$ disjoint from $L \subset \Sigma$. The proof will conclude by showing that Hamiltonian functions satisfying the sufficient conditions exist.

\medskip

{\it Displacing $I$ from $L$.}

\medskip

We can find two Hamiltonian functions $a(z_2)$ and $b(z_2)= a(-z_2)$, with compact support in $D(1+t)$ whose corresponding time-1 flows, denoted by $\phi_a$ and $\phi_b$ respectively, satisfy   $\phi_a({\rm proj}_{z_2}(A)) \cap \{x_2 =0\} = \emptyset$ and  $\phi_b({\rm proj}_{z_2}(B)) \cap \{x_2 =0\} = \emptyset$. This is because ${\rm proj}_{z_2}(A)$ and ${\rm proj}_{z_2}(B)$ are compact subsets of $\{x_2 \le 0\} \cup E^+$ and $-(\{x_2 \le 0\} \cup E^+)$, respectively, and both sets have area $(1+t)/2$. In fact, as $A$ and $B$ are compact, choosing $\delta$ smaller if necessary we may assume $a$ and $b$ have compact support in $D(1+t-
\delta)$. Note that one can choose $a$ such that $0 \le a \le t/2 - \delta$. In addition, we can let $a=0$ on $\{x_2=0\} \setminus D(1)$, and make sure the flow $\phi_a$ of $\{x_2=0\} \cap D(1)$ remains in $D(1)$.

\vskip 5pt

Next, let $\chi(x)$ be such that $\chi(x) =0$ if $x < 0$, $\chi(x) =1$ if $x > \delta$, and $0 \le \chi'(x) < 1/ \delta$.  
We consider the  Hamiltonian function $F:{\mathbb C}^2  \to \RR$ defined by 
\begin{equation} \label{def-ham-func-before-change}
F(x_1,y_1,x_2,y_2) = \chi(x_1)a(x_2,y_2) + \chi(-x_1)b(x_2,y_2).
\end{equation}
Note that the 
Hamiltonian flow 
satisfies $\phi_F(I) \cap L = \phi_F (A \sqcup B) \cap L = \emptyset$.
More generally we have the following.
\begin{lemma}\label{moveAB}
Suppose $H: {\mathbb C}^2 \to \RR$ satisfies
$H=F$ in $\Sigma$ and in a neighborhood of $\Sigma \cap \{|x_1|< \delta \}$.
Then the time-1 flow satisfies $\phi_H(A \sqcup B) \cap L = \emptyset$.
\end{lemma}
\begin{proof}
Suppose $p = (z_1, z_2) \in A$ (the argument for points in $B$ is identical), then $x_1 > \delta$. The component of the Hamiltonian vector field $X_H$ in the direction $\frac{ \partial}{\partial x_1}$ is given by $\frac{ \partial H}{ \partial y_1}$, which vanishes when $x_1 = \delta$ as $H$ agrees with $F$. Also, as $p \in \Sigma$, when $x_1 > \delta$ we have $\frac{ \partial H}{ \partial x_1} =\frac{ \partial F}{ \partial x_1} = 0$. Therefore the flow $\phi_H$ of $p$ remains in $\Sigma$, and as the $z_2$-component of $X_F$ on $\Sigma \cap \{x_1 > \delta \}$ is independent of $x_1$, we see that $\phi_H$ also displaces $A$ from $L = \Sigma \cap \{x_2 =0\}$. 
\end{proof}

\medskip

{\it Controlling the flow of $J$.}

\medskip

Note that showing that $J = {\bf K} \cap \Sigma^c$ remains disjoint from $L$ under a Hamiltonian flow is equivalent to showing that the inverse flow applied to $L$ remains disjoint from $J$. Consider the Hamiltonian flow $\phi^s_{-F}$ generated by $-F$, so $\phi^1_{-F} = (\phi_F)^{-1}$. This flow has the following property: 

\begin{lemma}\label{Faction} Let $0 \le s \le 1$. With the above notations one has $$\phi^s_{-F}(L) \subset \Sigma \cup \{- \delta < x_1 < \delta, \, 0 \le {\rm sign}(x_1) y_1 < \frac{t}{2 \delta} -1, \, z_2 \in D(1) \}.$$
\end{lemma}

\begin{proof} Suppose $x_1 \ge 0$. Then
we have $$X_{-F} = \chi'(x_1)a(z_2) \frac{\partial}{\partial y_1}  + \chi(x_1) X_a.$$
For a point $(x_1,0,0,y_2) \in L$, set $p:=(x_1',y_1',x_2',y_2')$ to be its image under the flow of $X_{-F}$. Note that if $x_1 > \delta$, then $\chi'=0$ and in particular $p \in \Sigma$. Assume $0<x_1<\delta$. Note that if the projection ${\rm proj}_{z_2}(x_1,0,0,y_2) = (0,y_2) \notin D(1)$, then by our assumption on the Hamlitonian function $a$ one has $a(0,y_2)=0$, and thus $y_1'=0$, i.e., $p \in \Sigma$.  If $(0,y_2) \in D(1)$, then, by the assumptions  on $\chi$ and $a$,  one has that $0 \leq y_1' \leq {\frac {t} {2\delta}} -1$, and that $(x_2',y_2') \in D(1)$. 
This completes the proof of the lemma for $x_1 \geq 0$. A similar argument works for the case $x_1 <0$. 
\end{proof}

Let $\xi$ be a symplectomorphism of the $z_1$-plane mapping the region
$$\{ - \delta < x_1 < \delta, \, 0 \le {\rm sign}(x_1) y_1 < \frac{t}{2 \delta} -1 \}$$
into the disk $D(t) \setminus \{|x_1| > \delta \ {\rm and} \  |y_1| < \delta \}$. 
Furthermore, suppose that $\xi$ is the identity near $\{y_1=0\}$. 
%
Define $G:{\mathbb C}^2  \to \RR$ by $G = F \circ (\xi \times {\rm Id})^{-1}$. The following is a corollary of Lemma \ref{Faction}.

\begin{corollary} \label{Gaction} Let $0 \le s \le 1$. Then $$\phi^s_{-G}(L) \subset \Sigma \cup ( (\{ \pi|z_1|^2 < t \}\setminus \{|x_1| > \delta \ {\rm and} \  |y_1| < \delta \}) \cap B^4(1+t)).$$
\end{corollary}

\begin{proof} The proof follows immediately from Lemma~\ref{Faction} and the fact that the flow of $G$ is given by $\phi^s_G = (\xi \times {\rm Id}) \circ \phi^s_G \circ (\xi \times {\rm Id})^{-1}$.
\end{proof}

Our repositioning Lemma \ref{two} now implies that $\phi_G(J) \cap L = \emptyset$. Slightly more generally, let $U$ be a neighborhood of the union of the sets $  \phi^s_{-G}(L) $
for $0 \le s \le 1$. Then we have the following.

\begin{corollary} \label{Haction} Let $H: {\mathbb C}^2 \to \RR$ so that $H=G$ on $U$. Then $\phi_H(J) \cap L = \emptyset$.
\end{corollary}

\medskip

{\it Displacing ${\bf K}$.}

\medskip

Recall that the flow generated by a Hamiltonian function $H : \CC^2 \to \RR$ preserves the ball $B^4(1+t)$ provided it is constant on the characteristic circles in $\partial B^4(1+t)$. Summarizing our discussion above, combining Lemma \ref{moveAB}. Corollary \ref{Haction}, and the symplectomorphism from Step 1, give the following.

\begin{corollary}\label{extend} Suppose $H : \CC^2 \to \RR$ is such that
\begin{enumerate}
    \item $H$ is constant on characteristics of $\partial B^4(1+t)$,
    \item 
    $H=F$ in $\Sigma$ and in a neighborhood of $\Sigma \cap \{|x_1|< \delta \}$,
    \item $H=G$ on $ U$.
\end{enumerate}
Then, the time-1 map $\phi_H$ restricts to give a diffeomorphism of $B^4(1+t)$ which displaces ${\bf K} = I \sqcup J$ from $L$.    
\end{corollary}

It remains to show that such functions $H$ exist. We will show that the three conditions in Corollary~\ref{extend} are consistent.

\medskip

We start by showing the compatibility of the first two conditions. 
For this we need to show that for each Hopf circle on $\partial B^4(1+t)$, the restriction of $F$ to the intersection of the Hopf circle with $\Sigma$, and to a neighborhood of $\Sigma \cap \{|x_1| < \delta\}$, is constant.

\medskip

Note that there is a single characteristic, $\{ z_1 =0 \} \cap \partial B^4(1+t)$, lying entirely in $\Sigma$. As our functions $a$ and $b$ have compact support in $D(1+t)$, the function $F$ is identically $0$ on this circle.
The remaining characteristics intersect $\Sigma$ in exactly two points, say $(x_1,0,x_2,y_2)$ and $(-x_1,,0,-x_2,-y_2)$. Assume $x_1 \geq 0$ (a similar argument holds for $x_1<0$), note that
$$ F(x_1,0,x_2,y_2) = \chi(x_1)a(x_2,y_2) = \chi(-x_1) b(-x_2, -y_2) =  F(-x_1,0, -x_2,-y_2),$$
and so for each Hopf circle $F$ is constant on its intersection with $\Sigma$. Thus $F$ can be extended over $\partial B^4(1+t)$ to a function $H$ constant on the characteristics.

Now we consider characteristics intersecting the neighborhood $\{ |z_1| < \delta \}$ of $\partial B^4(1+t) \cap \Sigma \cap \{|x_1| < \delta\}$. These characteristics intersect $\Sigma$ inside the region $$\partial B^4(1+t) \cap \Sigma \cap \{|x_1| < \delta\} \subset \partial B^4(1+t) \cap \Sigma \cap \{ \pi|z_2|^2 > 1+t-\delta \}$$ where $F$ is identically $0$. Hence a function $H$ on $\partial B^4(1+t)$ which agrees with $F$ on $\Sigma$ and is constant on characteristics will be identically $0$ on $\{ |z_1| < \delta \}$. In particular $H$ agrees with $F$ on a neighborhood of $\Sigma \cap \{|x_1|< \delta \}$ and we have shown the compatibility of conditions 1 and 2.

\medskip 
 
Regarding condition 3, as $U$ is relatively compact in the ball, its only intersection with the boundary is on $\Sigma$. Also $F=G$ in a neighborhood of $\Sigma$. Thus we can define a smooth function $H$ simultaneously equal to $G$ on $U \setminus \Sigma$, equal to $F$ on $\Sigma$ and in a neighborhood of $\Sigma \cap \{|x_1|< \delta \}$, and equal to our extension of $F$ over $\partial B^4(1+t)$ which is constant on the characteristics.
In other words, one can find a smooth function $H$ as required.
This completes the proof of Theorem~\ref{thm-embedding-to-ball-with-Lagrangian-removed}.
\end{proof}

Equipped with Theorem~\ref{thm-embedding-to-ball-with-Lagrangian-removed}, we turn now to the final ingredient needed for the proof of Theorem~\ref{thm-complement-plane}. 

\begin{proposition} \label{prop-cylindrical-upper-bnd}
For any $t \in (0,1)$ and $n \ge 2$ 
  $$\overline{c}(B^{2n}(1) \setminus (P_t \times \CC^{n-2})) \le \frac{1+t}{2}.$$    \end{proposition}

\begin{proof}
We need to produce a symplectic embedding $$B^{2n}(1) \setminus (P_t \times \CC^{n-2}) \xhookrightarrow{\mathrm s} Z^{2n}(\frac{1+t}{2} + \delta)$$ for all $\delta>0$.
Given a compact ${\bf K} \subset B^4(1) \setminus P_t$, Theorem \ref{thm-embedding-to-ball-with-Lagrangian-removed} gives a symplectic embedding ${\bf K} \xhookrightarrow{\mathrm s} B^4(1+t) \setminus L$. Theorem 1.3 from \cite{SSVZ} says that there exists a symplectic embedding $B^4(1+t) \setminus L\xhookrightarrow{\mathrm s} Z^4(\frac{1+t}{2})$, and so
composing gives an embedding ${\bf K}\xhookrightarrow{\mathrm s}  Z^4(\frac{1+t}{2})$.
Next we observe that $B^4(1) \setminus P_t$ embeds into a compact subset of $B^4(1 + \delta) \setminus P_t$ for all $\delta>0$. Indeed, a suitable embedding is the restriction of a symplectic embedding $\Xi$ of $\CC^2$ into itself defined as follows. Write $\CC^2 = P_t \oplus Q_t$, where $Q_t$ is the symplectic complement of $P_t$. Let $\zeta$ be a $C^0$-small symplectic embedding $Q_t \setminus \{0\} \xhookrightarrow{\mathrm s}  Q_t \setminus B_{\delta'}$, where $B_{\delta'}$ is small neighborhood of the origin. Thus, the required embedding is $$\Xi = {\rm Id} \times \zeta :P_t \oplus Q_t \to P_t \oplus Q_t.$$
Compose the above embedding ${\bf K}\xhookrightarrow{\mathrm s}  Z^4(\frac{1+t}{2})$ with $\Xi$, we obtain an embedding
$$\phi: B^4(1) \setminus P_t \xhookrightarrow{\mathrm s} Z^4(\frac{1+t}{2} + \delta).$$
To conclude, we observe that
$$\phi \times {\rm Id}:(z_1, \dots, z_n) \mapsto (\phi(z_1,z_2), z_3, \dots, z_n)$$ satisfies $$\phi \times {\rm Id}((B^{2n}(1) \setminus (P_t \times \CC^{n-2}))) \subset \phi(B^4(1) \setminus P_t) \times \CC^{n-2} \subset Z^{2n}(\frac{1+t}{2} + \delta).$$ This completes the proof of the proposition.
\end{proof}

\begin{proof}[{\bf Proof of Theorem~\ref{thm-complement-plane}.}]
The proof follows immediately  by  combining Proposition~\ref{c3n}, Proposition~\ref{prop-cylindrical-upper-bnd}, and the unitary equivalence between the codimension-two linear space  $E_t$ and $P_t \times {\mathbb C}^{n-2}$. 
\end{proof}

\section{The Complement of Parallel Subspaces}
In this section we prove Theorem~\ref{Thm-union-of-planes} and Theorem~\ref{thm-GW-capacity-Gromov-width}.
As before, assume that $n>1$ and equip ${\mathbb C}^n \simeq {\mathbb R}^{2n}$ with coordinates $(z_1,\ldots,z_n) = (x_1,y_1,\ldots,x_n,y_n)$, and with the standard symplectic form $\omega = dx \wedge dy$. 

\medskip


The proof of Theorem~\ref{Thm-union-of-planes} is broken into two parts: obtaining an upper bound for the cylindrical capacity, and a lower bound for the Hofer-Zehnder capacity. For the upper bound 
we  need the following observation.
Let $K \subset {\mathbb R}^{2n} \simeq {\mathbb C}^n$ be a convex body, and denote by $\ehzcap(K)$ the Ekeland-Hofer-Zehder capacity associated with $K$, i.e., the minimal action among the closed characteristics on the boundary $\partial K$ (see, e.g., Section 1.5 in \cite{HZ}). 
Moreover, for $L>0$, let $A^L : {\mathbb C}^n \rightarrow {\mathbb C}^n$ be the linear map that takes $z_n$ to $L z_n$, and leaves $z_i$ fixed for $1\leq i \leq n-1$.

\begin{proposition}
	\label{upper_bound}
	For convex $K \subset \R^{2n} \simeq {\mathbb C}^n$ such that $K = -K$ one has
	$$\lim_{L \to \infty} \ehzcap(A^L K) \leq \ehzcap(K \cap \{z_n = 0\}).$$
\end{proposition}
Note that $K \cap \{z_n = 0\} \subset {\mathbb C}^{n-1}$, and hence its capacity is taken with respect to the standard symplectic form restricted to ${\mathbb C}^{n-1}$. 
\begin{proof}
	Assume without loss of generality that $K$ is also strictly convex and smooth.
	Recall that, by Clarke's dual action principle  (see, e.g., Section 1.5 of~\cite{HZ1}), one has that for every convex body $T \subset {\mathbb R}^{2n}$
	\begin{align}
	\label{clark_eq}
	\ehzcap(T) = \min_{\eta \in \mathcal{E}_n, \mathcal{A}(\eta) = 1} \frac{\pi}{2} \int_0^{2\pi} h_T^2(\deta(t)) dt, 
	\end{align}
	where $\mathcal{E}_n = \{\eta \in W^{1,2}(S^1,\R^{2n}) | \int_0^{2\pi} \eta(t)dt = 0 \}$, the function  $h_T$ is the support function of $T$ defined by $h_T(u) := \sup \{ \langle x,u \rangle \, | \, x \in T \}$, and $\mathcal{A}(\eta)$ is the symplectic action of $\eta$.
	Moreover, one has that for $\eta$, a minimizer of \eqref{clark_eq}, and some $\lambda \in \R^+$ and $b \in \R^{2n}$, the orbit $\gamma(t) := \lambda J \eta(t) + b$ is a closed characteristic on the boundary of $T$ with minimal action.
	Denote by $\eta_1$ a minimizer of \eqref{clark_eq} for the body $K \cap \{z_n = 0\}$, and by $\gamma_1$ the corresponding closed characteristic.
In order to bound the capacity of $A^L K$, we consider the loop $\eta_2$ defined by 
$$\deta_2(t) = \|\deta_1(t)\| 
 \displaystyle \frac{\displaystyle  h_{K \cap \{z_n = 0\}}(n_{K \cap \{z_n=0\}}(\gamma_1(t)))}{\displaystyle h_{A^L K}(n_{A^L K}(\gamma_1(t)))} n_{A^L K}(\gamma_1(t)),$$ where $n_K(\cdot)$ is the unit outer normal to $K$ at a point.
	Recall that as $K = -K$ one has that $\gamma_1$ can be assumed to be centrally symmetric (see Corollary 2.2. in~\cite{AK}), and hence $\eta_2$ is a closed loop.
	Define $\alpha_1(t)$ and $\alpha_2(t)$ so that
	$$ n_{K}(\gamma_1(t)) = \displaystyle \frac{(n_{K \cap \{z_n=0\}}(\gamma_1(t)),\alpha_1(t),\alpha_2(t))}{\sqrt{1+\alpha_1(t)^2+\alpha_2(t)^2}} .$$
	Note that since $\gamma_1 \in \{z_n=0\}$,
        \begin{align*}
	n_{A^L K}(\gamma_1(t)) &= \frac{A^{1/L}n_{K} (A^{1/L} \gamma_1(t))}{\|A^{1/L}n_{K} (A^{1/L} \gamma_1(t))\|} \\
        &= \frac{(n_{K \cap \{z_n=0\}}(\gamma_1(t)),\frac{\alpha_1(t)}{L},\frac{\alpha_2(t)}{L})}{\sqrt{1+\alpha_1(t)^2+\alpha_2(t)^2}\|A^{1/L}n_{K} (A^{1/L} \gamma_1(t))\|}  . 
        \end{align*}
        Moreover, since  for every $p \in \partial K$ one has that 
        $h_K(n_K(p)) = \langle n_K(p),p \rangle$, 
         \begin{align*}
        \frac{n_{A^L K}(\gamma_1(t))}{h_{A^L K}(n_{A^L K}(\gamma_1(t)))} 
        &= \frac{(n_{K \cap \{z_n=0\}}(\gamma_1(t)), \frac{\alpha_1(t)}{L}, \frac{\alpha_2(t)}{L})}{\langle (\gamma_1(t),0,0) , (n_{K \cap \{z_n=0\}}(\gamma_1(t)), \frac{\alpha_1(t)}{L}, \frac{\alpha_2(t)}{L}) \rangle } \\
        &= \frac{(n_{K \cap \{z_n=0\}}(\gamma_1(t)), \frac{\alpha_1(t)}{L}, \frac{\alpha_2(t)}{L})}{h_{K \cap \{z_n = 0\}}(n_{K \cap \{z_n=0\}}(\gamma_1(t)))} .
        \end{align*}
	Next, since $\deta_1(t)$ is parallel to $n_{K \cap \{z_n=0\} }(\gamma_1(t))$, one has
	$$ \deta_2(t) = (\deta_1(t),\frac{\|\deta_1\|\alpha_1(t)}{L},\frac{\|\deta_1\|\alpha_2(t)}{L}). $$
This together with the definition of the support function implies that 
	$$ h_{K \cap \{z_n=0\} }(\deta_1(t)) = \langle \deta_1(t) , \gamma_1(t) \rangle = \langle \deta_2(t) , \gamma_1(t) \rangle = h_{A^L K}(\deta_2(t)) .$$
	Note that one has $\mathcal{A}(\eta_2) = \mathcal{A}(\eta_1) + \frac{\|\deta_1\|^2}{L^2} \mathcal{A}(\alpha)$, where $\dot{\alpha}(t) = (\alpha_1(t),\alpha_2(t))$.
	As $\mathcal{A}(\alpha)$ is bounded, one gets that $\mathcal{A}(\eta_2) \to \mathcal{A}(\eta_1)=1$ as $L \to \infty$. Using~\eqref{clark_eq} and normalizing $\eta_2$ appropriately, complete the proof.
\end{proof}

\medskip

Recall that $\Sigma_\eps^t$ is a linear image 
of the set of codimension-two subspaces
\begin{align} \label{Sec3-family-of-hyperplanes}
\Sigma_\eps := \bigcup \{(z_1,z_2,\ldots,z_n) \in \C^n : z_n \in \eps \Z^{2} \},
\end{align}
with K\"ahler angle $t$, i.e., one has  $|\omega(n_1,n_2)| = t$, where $n_1,n_2$ are the unit outer normals to the subspaces in $\Sigma^t_{\eps}$.

\begin{proposition} \label{prop-upper-bound-family-hyperplanes}
 	For any $\delta > 0$ and $t \in (0,1)$, there exists $\eps>0$  such that
	$$ \overline c(B^{2n}(1) \setminus \Sigma_\eps^t) \leq t + \delta .$$  
\end{proposition}

\begin{proof}
Let $0<t<1$, and consider the  symplectic  matrix (cf. Example 2.2. in \cite{HKHO}) 
	\[
	M_t =\left(\begin{array}{c|c}
	 \mathds{1}_{{\mathbb R}^{2n-4} } & \bigzero\\
	\hline \\*
	\bigzero & 
	\begin{array}{cccc}
		\frac{t}{\sqrt{1-t^2}} & 0 & 0 & -1 \\
		0 & \frac{\sqrt{1-t^2}}{t} & 0 & 0 \\
		0 & -1 & \frac{t}{\sqrt{1-t^2}} & 0 \\
		0 & 0 & 0 & \frac{\sqrt{1-t^2}}{t}
	\end{array}\\
	\\
	\end{array}\right).
	\]
	It follows from the proof of Theorem 1.3 in \cite{HKHO} that for every $L > 1$
	$$ \overline{c}(M_t B^{2n}(1) \setminus \Sigma_\eps) \leq \left( 1 + \frac{\sqrt{2} \eps L}{\lambda(A^L K)} \right)^2 \overline{c}(A^L M_t B^{2n}(1)), $$
 where $\lambda(K) := \inf_{u \in \bS^{2n-1}} h_K(u)$, and $\Sigma_\eps$ is the family of  complex planes~\eqref{Sec3-family-of-hyperplanes}.
	Since $A^L M_t B^{2n}(1)$ is a centrally symmetric ellipsoid, Proposition \ref{upper_bound} implies that for every (normalized) symplectic capacity $c$ one has
	$$ \lim_{L \to \infty} c(A^L M_t B^{2n}(1)) \leq c(M_t B^{2n}(1) \cap \{z_n=0\}). $$
A direct computation shows that $M_t B^{2n}(1) \cap \{z_n=0\}$ is linearly symplectomorphic to the symplectic ellipsoid $E(1,\ldots,1,t) = \left\{ \pi \left( \sum_{i=1}^{n-2} |z_i|^2 + \frac{|z_{n-1}|^2}{t} \right) \leq 1 \right\}$ which has capacity $t$. 
Note that the normals to $M_t^{-1} \Sigma_\eps$ (in the $(\overline x, \overline y)$-coordinate system) are 
	$$ n_1 = (0,\ldots,0,0,-\sqrt{1-t^2},t,0), \quad\quad n_2 = (0,\ldots,0,0,0,0,1), $$
	and hence $\omega(n_1,n_2) = t$.
	Overall, for any $\delta > 0$ and $0<t<1$, there exist $L \gg 1$ and $\eps > 0$ such that 
	$$ \overline{c}(B^{2n}(1) \setminus M_t^{-1} \Sigma_\eps) \leq t + \delta, $$
and the proof of the proposition is thus complete. 
\end{proof}

We turn now to obtain the required lower bound for the Hofer-Zehnder capacity. 
To this end we shall need the following

\begin{proposition} \label{lem-lower-bound-intersection}
Let $t \in (0,1)$ and $\varepsilon > 0$, and 
let $A$ be a symplecitc matrix such that $A\Sigma^t_\eps = \Sigma_\eps$. Then for every $\delta >0$ there is a symplectic embedding
$$ (1-\delta) (AB^{2n}(1) \cap W^{2n}) 
 \xhookrightarrow{\mathrm s}  B^{2n}(1) \setminus \Sigma^t_{\varepsilon},
$$ 
  where
	 $$ W^{2n} := \left( \{x_n = 0, y_n=0\} \cap A B^{2n}(1) \right) \times {\rm span}\{x_n,y_n\}. $$
\end{proposition}

\begin{proof}
    The idea of the proof is to ``push"
  $\Sigma_\eps$ to a small neighborhood of the boundary of $A B^{2n}(1) \cap W^{2n}$ via a symplectic isotopy $\phi_t$ (see Figure \ref{figure-hyperplanes-push}). 
Then, $\phi^{-1}$ will give the required embedding.	More precisely, let $\{u_i\}$ be the set of all the directions connecting points in $\eps \Z^2 \setminus \{0\} \cap {\rm proj}_{z_n} A B^{2n}(1)$ to the origin. For every $v \in \{u_i\}$, set $l_v := \{\lambda v : \lambda > \delta_{v} \}$ such that $\cup_i \{l_{u_i}\} $ cover all the points in $\eps \Z^2 \setminus \{0\}$ (see Figure \ref{figure_push_points}).
	Next,  consider 
	 $$ W^{2n}_v :=  \left( \{x_n = 0, y_n=0\} \cap A B^{2n}(1) \right) \times l_v, $$
	 and define a function $\alpha_v: A B^{2n}(1) \cap \{x : \langle x,Jv \rangle = 0\} \to \R$ such that it vanishes outside $W^{2n}_v$, is equal to $1$ in the interior of $W^{2n}_v$, and with some smooth cut-off in between (see Figure \ref{figure_single_line}).
	 Define a smooth function $H_v(x)$ such that it equals $-\alpha_v(x) \langle x, Jv \rangle $ in a small neighborhood of $ W^{2n}_v \subset \{x : \langle x,Jv \rangle = 0\} $, and vanishes outside a slightly larger  neighborhood  (by choosing the former neighborhood small enough, one can make sure that the supports of $H_{v_i}$ are disjoint). Note that the Hamiltonian vector field is equal to $X_{H_v}(x) = \alpha_v(x) v$ for $x$ such that $\langle x,Jv \rangle = 0$ and $\langle x,v \rangle > \delta_v $. Hence, the Hamiltonian function  $\sum_i H_{v_i}$ generates  the required flow. Note that in a similar way one can also push the codimension-two hyperplane passing through the origin in $\Sigma_{\eps}$ to the boundary of $AB^{2n}(1) \cap W^{2n}$. This completes the proof of the lemma. 
\end{proof} 

\begin{figure}[h]
\centering
	\includegraphics[width=0.45\textwidth]{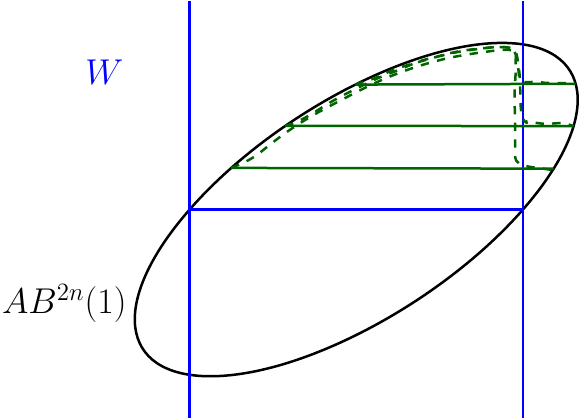}
 \caption{Pushing the hyperplanes arbitrarily close to the boundary.}
 \label{figure-hyperplanes-push}
 \end{figure}

 \begin{figure}[H]
\centering
\begin{minipage}{.45\linewidth}
 \includegraphics[width=\linewidth]{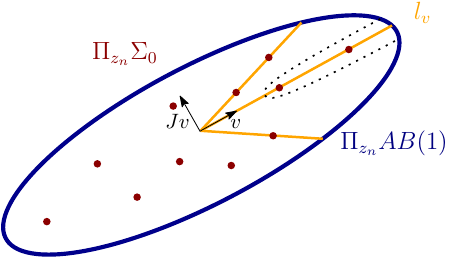}
 \caption{The line $l_v$ together with the support of the function $H_v$.}
 \label{figure_push_points}
\end{minipage}
\hspace{.05\linewidth}
\begin{minipage}{.45\linewidth}
 \includegraphics[width=0.9\linewidth]{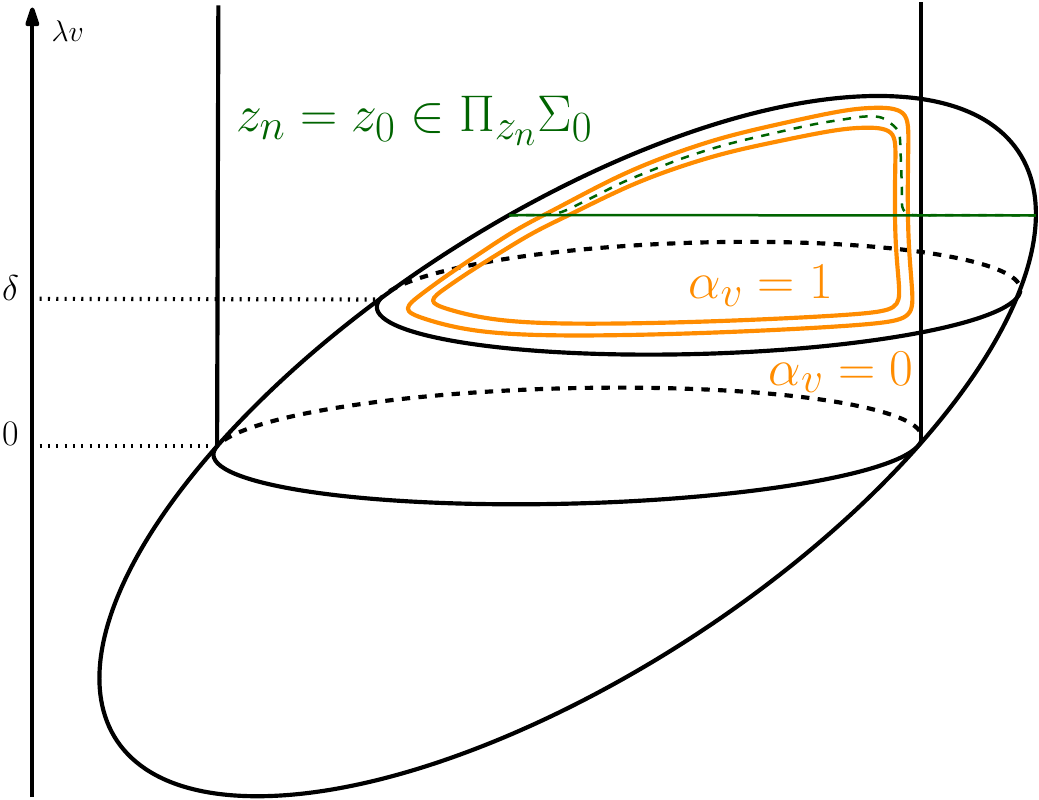}
 \caption{The domain $W^{2n}_v$ in the $(2n-1)$-space $\{\langle x,Jv \rangle = 0\}$.
 The flow of $H_v$ pushes $\{ z_n = z_0 \}$ close to the boundary.}
 \label{figure_single_line}
\end{minipage}
\end{figure}

From Proposition~\ref{lem-lower-bound-intersection} it follows that in order to obtain a lower bound for the capacity of $B^{2n}(1) \setminus \Sigma^t_{\varepsilon}$ it is enough to find a lower bound for the capacity of the intersection $AB^{2n}(1) \cap W^{2n}$, which is a convex domain in ${\mathbb R}^{2n}$. We will show below that the minimal action capacity of this domain equals $t$, which is suffices for the proof of Theorem~\ref{Thm-union-of-planes}. 
\begin{remark} {\rm In light of the well-known conjecture that all symplectic capacities coincide on convex domains in the classical phase space, one expects that the Gromov width of the intersection 
$AB^{2n}(1) \cap W^{2n}$ also equals $t$. This, in turn, would imply that Theorem~\ref{Thm-union-of-planes} holds for the Gromov width as well. However, we are only able to show that the Gromov width of the above  intersection is bounded below by $t-0.07$ as stated in    Theorem~\ref{thm-GW-capacity-Gromov-width}.
}
\end{remark}

\begin{proposition} 
\label{prop-EHZ-capacity-intersection} 
With the above notations one has
$$\ehzcap(AB^{2n}(1) \cap W^{2n}) = t. $$
\end{proposition}

\begin{proof}
We start with the 4-dimensional case.  Based on the definition of the Ekeland-Hofer-Zehnder capacity and the fact that $A$ is a symplectic matrix, it suffices  to show that any closed characteristic with minimal action on the boundary of $B^{4}(1) \cap A^{-1}W^4$ has action $t$.
For this end,
without loss of generality, one can choose $A$ to be the matrix 
\begin{align} \label{matrix-A-in-int-lemma}
	A^{-1} =\frac{1}{\sqrt{2}}\left(
	\begin{array}{cccc}
	\frac{\sqrt{1+t}}{\sqrt{t}} & 0 &  \frac{\sqrt{1-t}}{\sqrt{t}} & 0 \\
	0 &  \frac{\sqrt{1+t}}{\sqrt{t}} & 0 & -\frac{\sqrt{1-t}}{\sqrt{t}}  \\
	\frac{\sqrt{1-t}}{\sqrt{t}} & 0 &  \frac{\sqrt{1+t}}{\sqrt{t}} & 0 \\
	0 & -\frac{\sqrt{1-t}}{\sqrt{t}} & 0 & \frac{\sqrt{1+t}}{\sqrt{t}}  \\
	\end{array}\right).
\end{align}
In this case the base of the cylinder $A^{-1}W^4$ is spanned by 
$$v_1 := \frac{1}{\sqrt{2}}(\sqrt{1+t},0,\sqrt{1-t},0) \ {\rm and} \ v_2 := \frac{1}{\sqrt{2}} (0,\sqrt{1+t},0,-\sqrt{1-t}).$$ Complete $v_1,v_2$ into an orthonormal basis with 
$$n_1 := \frac{1}{\sqrt{2}}(0,\sqrt{1-t},0,\sqrt{1+t}) \ {\rm  and} \ n_2 := \frac{1}{\sqrt{2}}(-\sqrt{1-t},0,\sqrt{1+t},0).$$
Denote  $S_1 := \partial B^{4}(1) \cap A^{-1}W^4$ and $S_2 := B^{4}(1) \cap \partial A^{-1}W^4$, and note that $\partial(B^{4}(1) \cap A^{-1}W^4)  = S_1 \cup S_2$. We classify the closed characteristics on the boundary $\partial(B^{4}(1) \cap A^{-1}W^4)$ by how many time they ``visit" the sets $S_1$ and $S_2$. More precisely, note that a closed characteristic which lies entirely in $S_1$ has action $1$, and a closed characteristic which is entirely in $S_2$ has action $t$. Moreover, such a characteristic in $S_2$ exists in the subspace spanned by $v_1,v_2$.
The other options include closed characteristics that pass between $S_1$ and $S_2$ and vice versa (maybe several times), and closed characteristics that stay in the intersection $S_1 \cap S_2$ for all time. We analyse the latter two options below. We remark that from the proof below it follows that any closed characteristic spending a non-discrete time in $S_1 \cap S_2$ must remain within this intersection indefinitely.

\medskip
We start with some general observations regarding the Reeb dynamics that will be useful later on.
Observe that the characteristics in $S_1$ are moving along two centred circles in the $z_1$ and $z_2$-coordinate planes respectively, in the same direction and the same angular speed. The sum of the areas enclosed by the two circles is $1$.
On the other hand, a direct computation shows that the characteristics in $S_2$ are moving along two non-centred circles in the $z_1$, $z_2$-coordinate planes respectively, in opposite directions, and with the same angular speed. The areas enclosed by these circles are $\frac{1+t}{2}$ and $\frac{1-t}{2}$, respectively.
\begin{figure}[h]
\centering
\begin{minipage}{.4\linewidth}
  \includegraphics[width=\linewidth]{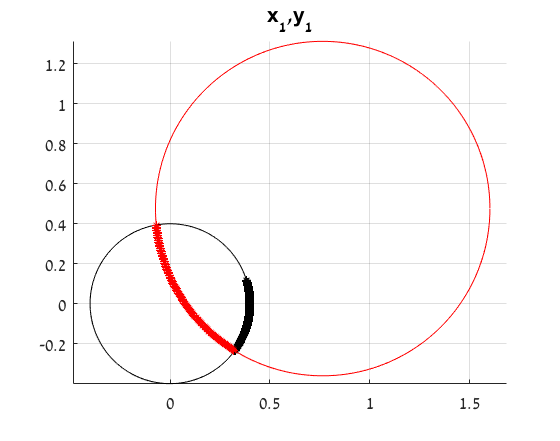}
\end{minipage}
\hspace{.01\linewidth}
\begin{minipage}{.4\linewidth}
  \includegraphics[width=\linewidth]{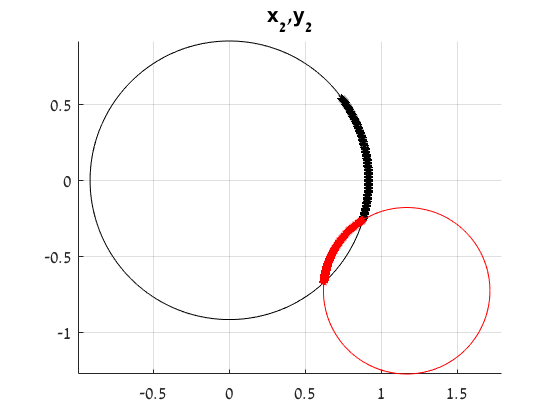}
\end{minipage}
\caption{The projections to the $z_1,z_2$ planes of the closed characteristic of $B^4(1)$ in black, and the closed characteristic of $A^{-1}W^4$ in red. The bold part is the section of the characteristic contained in $B^4(1) \cap A^{-1}W^4$.}
\end{figure}

\medskip

Using the fact that $p \in A^{-1}W^4$ if and only if 
$$\langle p , Jv_1 \rangle^2 + \langle p , Jv_2 \rangle^2 \leq \frac{t^2}{\pi}, $$
one can write
\begin{align*}
S_1 &= \{ x_1 J n_1 + x_2 J n_2 + x_3 J v_1 + x_4 J v_2 \, | \, \sum_{i=1}^4 x_i^2 = \frac{1}{\pi}, x_3^2+x_4^2 \leq \frac{t^2}{\pi} \}  \\
S_2 &= \{ x_1 J n_1 + x_2 J n_2 + x_3 J v_1 + x_4 J v_2 \, | \, \sum_{i=1}^4 x_i^2 \leq \frac{1}{\pi}, x_3^2+x_4^2 = \frac{t^2}{\pi} \}. 
\end{align*}
Note that for a point $p \in S_1$, the (not normalized) outer normal equals $p$, and the characteristic direction is $J p$. On the other hand, for $p \in  S_2$, the characteristic direction is $J n_{A^{-1}W^4}(p) $, where 
$$n_{A^{-1}W^4}(p) = \langle p, Jv_1 \rangle Jv_1 + \langle p, Jv_2 \rangle Jv_2 $$ 
is the outer normal to $A^{-1}W^4$ at $p$. From now on we set $x_1,x_2,x_3,x_4$ for a point $p$ so that
\begin{align*}
p = x_1 Jn_1 + x_2 Jn_2 + x_3 Jv_1 + x_4 Jv_2.
\end{align*}
Moreover, a point $p$ in $A^{-1}W^4$ can be written as $p = \alpha_1 v_1 + \alpha_2 v_2 + \alpha_3 J n_1 + \alpha_4 J n_2$, where $\alpha_1^2 + \alpha_2^2 \leq \frac{1}{\pi}$.
One can check that 
\begin{align*}
\alpha_3 = \frac{\langle p, n_2 \rangle}{t} = -\frac{\sqrt{1-t^2}}{t}x_4 + x_1, \quad\quad \alpha_4 = -\frac{\langle p, n_1 \rangle}{t} = \frac{\sqrt{1-t^2}}{t}x_3 + x_2.
\end{align*} 
A direct computation shows that the projection of a Hopf circle passing through a point $(z_1,z_2)$ on the plane spanned by $Jv_1, Jv_2$ is an ellipse with area given by $|\pi \|z_1\|^2 - \frac{1-t}{2} |$.
Denote the projection of the Hopf circle on $z_i$ by $\gamma_i$, and the projection to the $Jv_1,Jv_2$ plane by $\widetilde{\gamma}$.
Denote the projection of the characteristics of $A^{-1}W^4$ on $z_i$ by $\eta_i$ (non-centred circles) and the projection on the plane spanned by $Jv_1,Jv_2$ by $\widetilde{\eta}$ (centred circle of radius $\frac{t}{\sqrt{\pi}}$). 

\begin{figure}[h]
\centering
\includegraphics[width=0.5\linewidth]{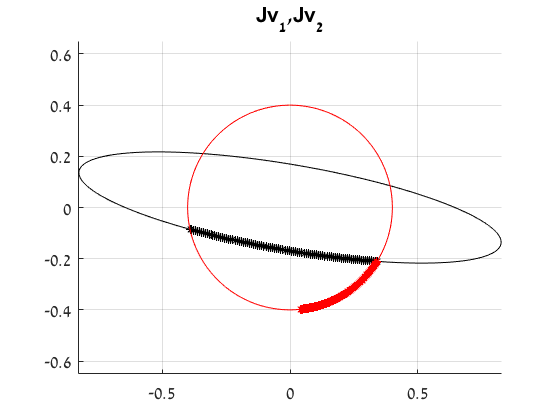}
\caption{The projection of the characteristics to the $(Jv_1,Jv_2)$-plane.}
\end{figure}

\medskip

We  analyse first closed characteristics entirely contained in the intersection of $S_1$ and $S_2$. 
Note that for $p \in S_1 \cap S_2$ the characteristic direction is  a non-negative linear combination of the two outer normals  $Jp$ and $Jn_{A^{-1}W^4}(p)$, denoted by $s := \alpha Jp + \beta Jn_{A^{-1}W^4}(p)$ for some $\alpha,\beta \geq 0$.
In order for the characteristic to stay in the intersection, 
 $s$ needs to be tangent to $S_1 \cap S_2$, meaning $\langle s, p \rangle = \langle s, n_{A^{-1}W^4}(p) \rangle = 0$, which is equivalent to $\langle p , J n_{A^{-1}W^4} (p) \rangle  = 0$.
This condition, together with the fact that $x_1^2 + x_2^2 = \frac{1-t^2}{\pi}, x_3^2+x_4^2 = \frac{t^2}{\pi}$, implies that 
$$ x_1 = \lambda x_4, \quad x_2 = -\lambda x_3, \quad\quad \text{where, } \quad \lambda = \pm \frac{\sqrt{1-t^2}}{t}. $$
We start with the case $\lambda = \frac{\sqrt{1-t^2}}{t}$. 
In order for the characteristic to remain in the intersection one also needs to require that $\langle s, Jn_1 \rangle = \lambda \langle s, Jv_2 \rangle$, and $\langle s, Jn_2 \rangle = -\lambda \langle s, Jv_1 \rangle$. This condition implies that $\alpha = 0$, i.e. the velocity is only ``coming" from $S_2$. 
The point $p$ in Euclidean coordinates is of the form 
$$ \left( \frac{\sqrt{1+t}}{\sqrt{2}t} x_4, -\frac{\sqrt{1+t}}{\sqrt{2}t} x_3, \frac{\sqrt{1-t}}{\sqrt{2}t} x_4, \frac{\sqrt{1-t}}{\sqrt{2}t} x_3 \right), $$
and in this case $\alpha_3 = \alpha_4 = 0$.
This is the same closed characteristic as in the case when the characteristic is contained in $S_2$, which has action $t$.
In addition, we claim that moving in the direction of the Hopf circle (and possibly leaving the intersection) is not possible. Indeed, recall that the projection of the Hopf circle $\widetilde{\gamma}$ passing through $p$ is an ellipse, which is tangent to $\widetilde{\eta}$, a circle of radius $t/ \pi$ (and area $t^2$). As the area of the ellipse is $t$, it must contain the circle, and hence the Hopf circle intersects  $A^{-1}W^4$ only in the tangency points.
This observation also implies that a characteristic cannot get to $p$ from $S_1$.

\medskip

We turn to case that $\lambda = - \frac{\sqrt{1-t^2}}{t}$.
The conditions  $\langle s, Jn_1\rangle = \lambda \langle s,  Jv_2 \rangle$ and $ \langle s, Jn_2\rangle = - \lambda \langle s, Jv_1 \rangle$ imply that $\beta = \left(\frac{1}{2t^2} -2\right)\alpha$. As $\alpha, \beta > 0$, we get that this case holds only for $t<\frac{1}{2}$. In this case,
$$ s = \frac{\alpha}{2t} \left( -x_2 Jn_1 + x_1 Jn_2 - x_4 Jv_1 + x_3 Jv_2 \right).$$ 
This creates a simultaneous circular movement in the $Jn_1, Jn_2$ and $Jv_1, Jv_2$ planes.
One can check that the projection of this orbit to the $z_1,z_2$ coordinates are also centred circles (rotating in opposite directions).
The point $p$ in Euclidean coordinates is of the form
$$ \left( \frac{\sqrt{1+t}}{\sqrt{2}}(2-\frac{1}{t}) x_4, - \frac{\sqrt{1+t}}{\sqrt{2}}(2-\frac{1}{t}) x_3, - \frac{\sqrt{1-t}}{\sqrt{2}}(2+\frac{1}{t}) x_4, -\frac{\sqrt{1-t}}{\sqrt{2}}(2+\frac{1}{t}) x_3 \right). $$
The symplectic action is thus
$$ \pi\|z_2\|^2 - \pi\|z_1\|^2 = t(3-4t^2). $$
Since $t<\frac{1}{2}$, the action is larger then $t$.
Similarly to the case of $\lambda = \frac{\sqrt{1-t^2}}{t}$, the projection of the Hopf circle, $\widetilde{\gamma}$, is an ellipse of area $t(1-2t^2)$, which is larger than $t^2$ (as $t<\frac{1}{2}$). Hence the ellipse is tangent to the circle of radius $t$ from the outside, which means that the direction of the Hopf circle does not interact with this characteristic.

\medskip 

It remains to consider the case of a characteristic which alternates between $S_1$ and $S_2$.
Assume without loss of generality that the starting point of the characteristic is in $S_1 \cap S_2$ and it moves along $S_2$ until it hits again the intersection.
We claim that the norms of the $z_1$ and $z_2$ coordinates are the same at these two points (before and after moving in $S_2$).
To show this, we calculate the intersection of a characteristic of $S_2$ with the boundary of the unit ball.
Recall that a possible representation of this characteristic is 
$$ \alpha_1 v_1 + \alpha_2 v_2 + \alpha_3 Jn_1 + \alpha_4 Jn_2 $$
for fixed $\alpha_3,\alpha_4$ and $\alpha_1^2 + \alpha_2^2 = \frac{1}{\pi}$.
In Euclidean coordinates this becomes
\begin{align*}
     \frac{1}{\sqrt{2}} \Bigl( & \alpha_1 \sqrt{1+t}  + \alpha_3 \sqrt{1-t}, \alpha_2 \sqrt{1+t} + \alpha_4 \sqrt{1-t} ,  \\
     &  \alpha_1 \sqrt{1-t} + \alpha_3 \sqrt{1+t}, -\alpha_2 \sqrt{1-t} - \alpha_4 \sqrt{1+t} \Bigl).
\end{align*}
This implies that
\begin{align*}
\|z_1\|^2 &= \frac{1}{2} \left( \frac{1+t}{\pi} + (1-t)(\alpha_3^2 + \alpha_4^2) + 2 \sqrt{1-t^2} (\alpha_1 \alpha_3 + \alpha_2 \alpha_4) \right), \\
\|z_2\|^2 &= \frac{1}{2} \left( \frac{1-t}{\pi} + (1+t)(\alpha_3^2 + \alpha_4^2) + 2 \sqrt{1-t^2} (\alpha_1 \alpha_3 + \alpha_2 \alpha_4) \right) .
\end{align*}
In the intersection $\|z_1\|^2 + \|z_2\|^2 = \frac{1}{\pi}$, and hence
$$ \|z_1\|^2 = \frac{1 + t}{2\pi} - \frac{t(\alpha_3^2 + \alpha_4^2)}{2}, \quad\quad \|z_2\|^2 = \frac{1 - t}{2\pi} + \frac{t(\alpha_3^2 + \alpha_4^2)}{2}. $$
Since these expressions are independent of $\alpha_1,\alpha_2$, we get that $\|z_i\|$, for $i=1,2$, is the same before and after the movement in $S_2$.
In addition, we note that as one varies $\alpha_1$ and $\alpha_2$, the change in $\| z_1 \|$ and $\| z_2 \|$ is the same. Since $\|z_1\|^2 + \|z_2\|^2 \leq \frac{1}{\pi}$, this means that the movement along $\eta_i$ lies inside the disc enclosed by $\gamma_i$.
We continue along movement on $S_1$, which has the same projection to $z_2$ as $\gamma_2$.
Consider a closed characteristic which starts with movement in $S_2$ with angular change $\theta_1$ along $\eta_2$, then movement on $S_1$ with angular change $\tau_1$ (see Figure \ref{S1S2orbit}), 
continuing with movements that alternate between $S_2$ and $S_1$ with angular movements $\theta_2,\tau_2,\ldots,\theta_k,\tau_k$.
Denote by $\widetilde{\theta}_i$ the angular change along $\gamma_2$ which corresponds to $\theta_i$.
\begin{figure}[H]
\centering
\includegraphics[width=0.7\linewidth]{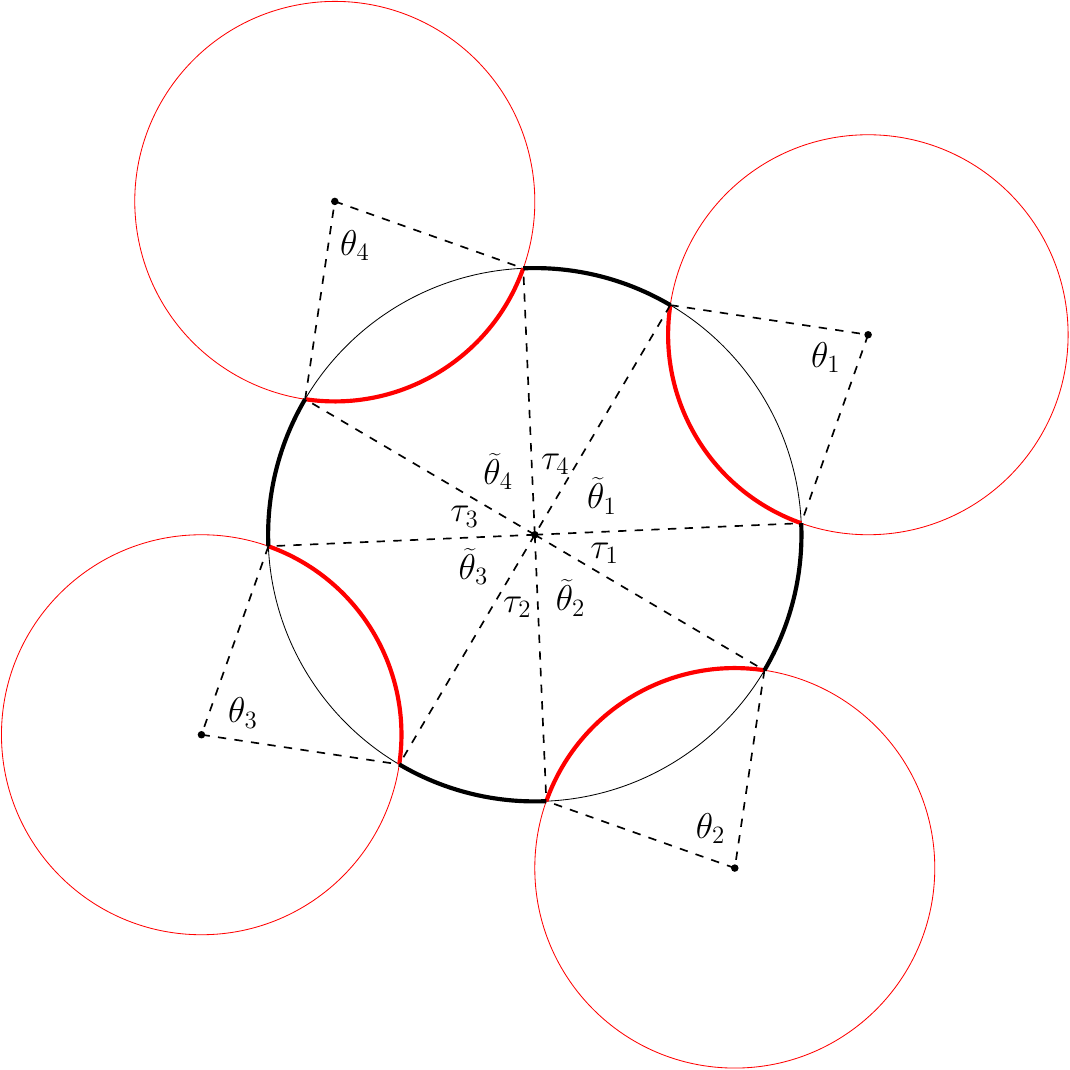}
\caption{The projection of a closed orbit alternating between $S_1$ and $S_2$ to the $z_2$ plane.}
\label{S1S2orbit}
\end{figure}
As the radius of $\eta_2$ (i.e., $\sqrt{{(1-t})/(2\pi)}$) is always smaller then the radius of $\gamma_2$ (i.e., $\|z_2\| $), we get that $\widetilde{\theta}_i < \theta_i$.
Hence the action of the loop is
$$ \mathcal{A} = \sum_{i=1}^k \frac{\theta_i}{2 \pi} t + \frac{\tau_i}{2 \pi} > \sum_{i=1}^k \frac{\widetilde{\theta}_i}{2 \pi} t + \frac{\tau_i}{2 \pi} > \left( \sum_{i=1}^k \frac{\widetilde{\theta}_i}{2 \pi} + \frac{\tau_i}{2 \pi} \right) t  \geq t, $$
where the last inequality is due to the fact that the loop is closed and the orientation of the loops does not change.
This completes the proof of the proposition in the 4-dimensional case. For a general dimension $n > 2$, note that one can assume that the symplectic matrix, which we now denote by $A_{2n}$ to distinguish it from the 4-dimensional case above, is of the form
\[
	A_{2n} =\left(\begin{array}{c|c}
	 \mathds{1}_{{\mathbb R}^{2n-4} } & \bigzero\\
	\hline \\*
	\bigzero & 
	A \\
	\\
	\end{array}\right),
	\]
where $A$ is the matrix~\eqref{matrix-A-in-int-lemma}.  In this case one has 
$$ A_{2n} B^{2n}(1) \cap W^{2n} = B^{2n-4}(1) \times_2 (A B^4(1) \cap W^4),$$
where $\times_2$ stands for the symplectic $2$-product defined more generally for two convex domains  $K_1 \in {\mathbb R}^{2n}$ and $K_2 \in {\mathbb R}^{2m}$ by
$$ K_1 \times_2 K_2 := \bigcup_{0 \leq s \leq 1} \Bigl (   (1-s)^{1/2}K_1 \times s^{1/2}K_2 \Bigr).$$
From Proposition 1.5 in~\cite{HKO} one has
\begin{align*}
    \ehzcap(B^{2n-4}(1) \times_2 (AB^4(1) \cap W^4)) = \min\{ \ehzcap(B^{2n-4}(1)),  \ehzcap( AB^4(1) \cap W^4 )\} = t,
\end{align*}
which completes the proof of the proposition.
\end{proof}

\begin{proof}[{\bf Proof of Theorem~\ref{Thm-union-of-planes}}]
The upper bound follows immediately from 
   Proposition~\ref{prop-upper-bound-family-hyperplanes}, and the lower bound follows from the combination of Proposition~\ref{lem-lower-bound-intersection},
Proposition~\ref{prop-EHZ-capacity-intersection}, and the well known fact that the Hofer-Zehnder capacity coincide with the minimal action capacity for convex domains. 
\end{proof}

We turn now to the proof of Theorem~\ref{thm-GW-capacity-Gromov-width}, which shows, roughly speaking, that the Gromov width of  $B^{2n}(1) \setminus \Sigma_\eps^t$ is bounded below by an almost linear function.

\begin{proof}[{\bf Proof of Theorem~\ref{thm-GW-capacity-Gromov-width}}] 
For $0 < t < 1$, consider the function $f(t)$ defined by 
 \begin{align} \label{the-lower-bound-function}
	 f(t) := \sqrt{2\left(\frac{1}{t^2}-1\right)\left(\sqrt{1-t^2}-1\right)+1}.
\end{align}
 A direct computation shows that $f(t) \geq t - 0.07$ (see Figure~\ref{graph-function-f}). 
	 \begin{figure}[H]
	 	\includegraphics[width=\textwidth]{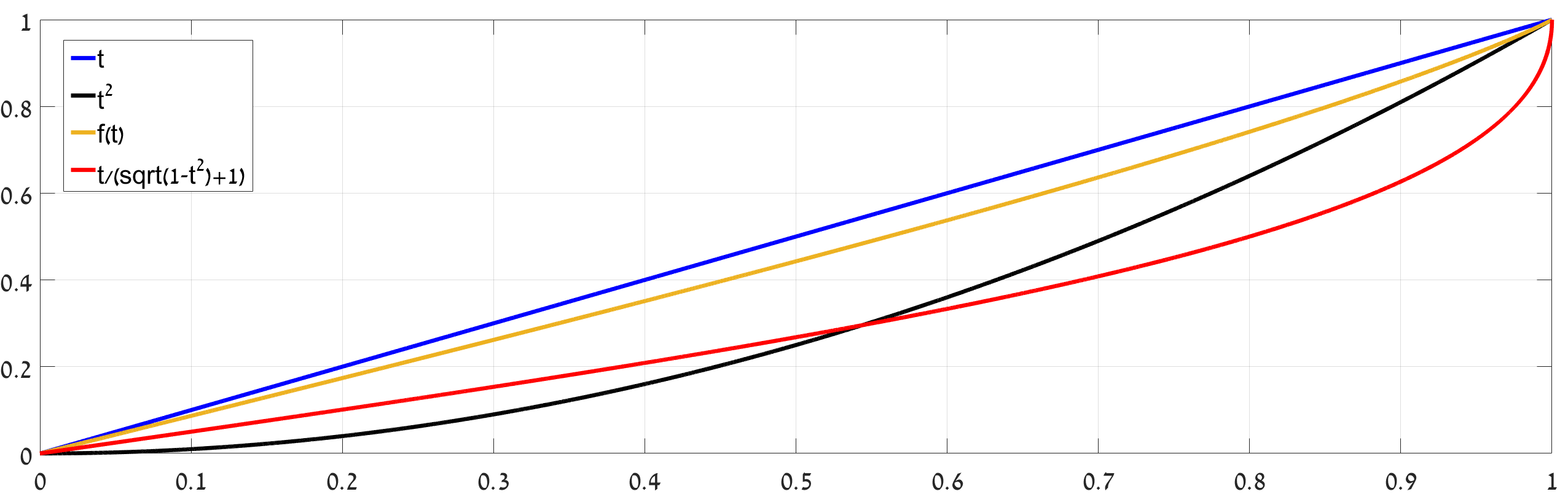}
	 	\caption{The different lower bounds together with the upper bound $t$.}
	 	\label{graph-function-f}
	 \end{figure}
From Proposition~\ref{lem-lower-bound-intersection} it follows that it is enough to show that 
$$ \gw(AB^{2n}(1) \cap W^{2n}) \geq f(t),$$
where $A$ is a symplectic matrix such that $A\Sigma^t_\eps = \Sigma_\eps$. 
We may assume without loss of generality that the outer normals to the hyperplanes in $\Sigma_\eps^t$ are given by
  $n_1 = (0,0,\ldots,0,\sqrt{1-t^2},0,-t,0)$ and $n_2 = (0,0,\ldots,0,0,0,0,1)$, written here in the $({\overline x},{\overline y})$-coordinate system. 
	Note that as symplectic matrix $A$ that maps  $\Sigma^t_\eps$ to $\Sigma_\eps$ one can now take
	\[
	A =\left(\begin{array}{c|c}
	 \mathds{1}_{{\mathbb R}^{2n-4} }& \bigzero\\
	\hline \\*
	\bigzero & 
	\begin{array}{cccc}
	\frac{1}{\sqrt{t}} & 0 & 0 & 0 \\
	0 & \sqrt{t} & 0 & \sqrt{\frac{1-t^2}{t}} \\
	-\sqrt{\frac{1-t^2}{t}} & 0 & \sqrt{t} & 0 \\
	0 & 0 & 0 & \frac{1}{\sqrt{t}} \\
	\end{array}\\
	\\
	\end{array}\right).
	\]


We first describe two immediate bounds for $ \gw(AB^{2n}(1) \cap W^{2n})$. The first comes from the largest Euclidean ball contained in the domain $B^{2n}(1) \cap A^{-1}W^{2n}$.
More precisely, let $E$ be the linear subspace such that $E\perp n_1$ and $ E \perp n_2$. Note that
	 $$ E = \C^{n-2} \times {\rm span}  \{ (t,0,\sqrt{1-t^2},0), (0,1,0,0) \},$$
	 and the corresponding symplectic orthogonal subspace is
	 $$ E^\omega = 0 \times {\rm span} \{ (0,-\sqrt{1-t^2},0,t) , (0,0,1,0) \}.$$
	 In addition, the orthogonal complement of $E^\omega$ is
	 $$ \left( E^\omega \right)^\perp = \C^{n-2} \times {\rm span} \{ (1,0,0,0), (0,t,0,\sqrt{1-t^2}) \}, $$
	 and since the orthogonal projection of $E \cap B^{2n}(1)$ to $\left( E^\omega \right)^\perp $ is
	 $$ \{(\lambda_1,\ldots,\lambda_{2n-4},\lambda_{2n-3}t,\lambda_{2n-2}t^2,0, \lambda_{2n-2} t\sqrt{1-t^2}) : \pi \sum_{i=1}^{2n-2} \lambda_i^2 \leq 1\}, $$
	 one has $B^{2n}(t^2) \subset B^{2n}(1) \cap A^{-1}W^{2n}$, which gives $ \gw(AB^{2n}(1) \cap W^{2n} ) \geq t^2$.

  \medskip 
The second lower bound is provided by 
  the largest Euclidean ball inside $A B^{2n}(1) \cap W$.
	 More precisely, note that $ \{x_n = 0, y_n=0\} \cap A B^{2n}(1) $ is of the form
	 \begin{align*}
	  &\{(x_1,y_1,\ldots,x_{n-2},y_{n-2},\frac{x_{n-1}}{\sqrt{t}},y_{n-1}\sqrt{t},0,0) | \pi \left( \sum_{i=1}^{n-2} x_i^2 + y_i^2 + \frac{x_{n-1}^2}{t^2} + y_{n-1}^2 \right) \leq 1 \} \\
	 &= \{(x_1,y_1,\ldots,x_{n-2},y_{n-2},x_{n-1},y_{n-1},0,0) | \pi \left( \sum_{i=1}^{n-2} x_i^2 + y_i^2 + \frac{x_{n-1}^2}{t} + \frac{y_{n-1}^2}{t} \right) \leq 1 \}, \\
	 \end{align*}
	 which is a $(2n-2)$-dimensional symplectic ellipsoid containing the ball $B^{2n-2}(t)$ of capacity $t$.
	 On the other hand, for the largest ball inside $A B^{2n}(1)$ one has
	 \begin{align*}
	 \text{inrad}(A B^{2n}(1)) = \min_{\|v\|=1/\sqrt{\pi}} \|A v\| = \frac{1}{\sqrt{\pi} \|A^{-1}\|_{\text{op}}} = \sqrt{\frac{t}{\pi\left(\sqrt{1-t^2} + 1\right)}}.
	 \end{align*}
	 Thus, $B^{2n}(\frac{t}{\sqrt{1-t^2} + 1}) \subseteq AB^{2n}(1) \cap W$, which gives $ \gw(B^{2n}(1) \setminus \Sigma^t_\eps ) \geq \frac{t}{\sqrt{1-t^2} + 1}$.
\vskip 5pt

Note that none of the above bounds dominates the other (see Figure~\ref{graph-function-f}). Moreover, in the first case  the constraint for embedding is coming from the intersection of the ball with the cylinder $A^{-1} W$, while in the second case, the constraint is due to the intersection of a ball with 
$A B^{2n}(1)$. 
A way to improve the two bounds above is to consider a symplectic linear image of the ball $S B^{2n}(r)$ such that the largest $r$ for which it fits in the ball $B^{2n}(1)$ is equal to the largest $r$ for which the image fits in the cylinder $A^{-1}W$.
	 For this, consider the following symplectic matrix for some parameters  $d_1,d_2  > 0$,
	 \[
	 S =\left(\begin{array}{c|c}
	 \mathds{1}_{{\mathbb R}^{2n-4}} & \bigzero\\
	 \hline \\*
	 \bigzero & 
	 \begin{array}{cccc}
	 d_1 & 0 & \sqrt{d_1 d_2-1} & 0\\
	 0 & d_2 & 0 & -\sqrt{d_1 d_2-1}\\
	 \sqrt{d_1 d_2-1} & 0 & d_2 & 0\\
	 0 & -\sqrt{d_1 d_2-1} & 0 & d_1
	 \end{array}\\
	 \\
	 \end{array}\right).
	 \]
	 Note that when $d_1=d_2=1$ this corresponds to the first embedding described above, and up to a unitary transformation (which does not change the embedding of the ball), there exist $d_1$ and $d_2$ which correspond to the second embedding, i.e., $A^{-1}$ has this form after multiplying with a unitary matrix. Now we choose the parameters $d_1$ and $d_2$ such that the projection of the image of the ball to $\left( E^\omega \right)^\perp$ is a symplectic ellipsoid, or, in other words, that the base of the relevant symplectic image of the cylinder is always a disc. 
  Moreover, as is often the case with similar optimization problems, we require that:
	 $$ \sup_{S B^{2n}(r) \subset B^{2n}(1)} r = \sup_{S B^{2n}(r) \subset A^{-1}W} r. $$
	 These two assumptions determine $d_1$ and $d_2$, and when plugging these solutions into $S$ we conclude that one can fit into $B^{2n}(1) \cap A^{-1}W$ a ball of capacity 
	 \begin{align*} 
        \sqrt{2\left(\frac{1}{t^2}-1\right)\left(\sqrt{1-t^2}-1\right)+1}.
	 \end{align*}
	 
\end{proof}	 

\begin{remark}
    Numerical tests suggest that the embedding above of a ball with capacity given by $\eqref{the-lower-bound-function}$ is the best embedding one can find using only linear symplectic maps. 
\end{remark}


\vskip10pt

\noindent Pazit Haim-Kislev \\
\noindent School of Mathematical Sciences, Tel Aviv University, Israel \\
\noindent e-mail: pazithaim@mail.tau.ac.il
\vskip 10pt

\noindent Richard  Hind \\
\noindent Department of Mathematics, University of Notre Dame, IN, USA. \\
\noindent e-mail: hind.1@nd.edu
\vskip 10pt

\noindent Yaron Ostrover \\
\noindent School of Mathematical Sciences, Tel Aviv University, Israel \\
\noindent e-mail: ostrover@tauex.tau.ac.il


\begin{thebibliography}{99}

\bibitem{AK} A. Akopyan, R. Karasev. {\it
Estimating symplectic capacities from lengths of closed curves on the unit spheres,} arXiv:1801.00242.





\bibitem{B} P. Biran. {\it Lagrangian barriers and symplectic embeddings,} Geom. Funct. Anal., 11 (2001),  407--464.

\bibitem{bh} O. Buse and R. Hind, {\it Symplectic embeddings of ellipsoids in dimension greater than four}, Geom. Topol. 15 (2011), 2091--2110.

\bibitem{CHLS} 
 Cieliebak, K., Hofer, H., Latschev, J.,  Schlenk, F. {\it Quantitative symplectic geometry}, in Recent Progress in Dynamics. Math. Sci. Res. Inst. Publ., 54 Cambridge University Press, Cambridge, 2007, 1--44.



\bibitem {Gr} Gromov, M. {\it Pseudo holomorphic curves in symplectic manifolds,} Invent.
Math., 82 (1985),  307--347.




\bibitem{HKHO} P. Haim-Kislev, R. Hind, Y. Ostrover. {\it On the existence of symplectic barriers}, Preprint. arXiv:2301.01822.

\bibitem{HKO} P. Haim-Kislev, Y. Ostrover.
{\it Remarks on symplectic capacities of $p$-products,}  Internat. J. Math. 34 (2023), no. 4, Paper No. 2350021.

\bibitem{HZ}  H. Hofer and E. Zehnder, {\it A new capacity for symplectic manifolds, Analysis, et cetera}, Academic Press, Boston, MA, 1990, pp. 405--427, 2011.

\bibitem{HZ1} H. Hofer and E. Zehnder, {\rm Symplectic invariants and Hamiltonian dynamics,}   Birkh\"auser Verlag, Basel, 1994.


\bibitem{MP} D. McDuff, L. Polterovich. {\it Symplectic packings and algebraic geometry. With an appendix by Yael Karshon}, Invent. Math. 115(3) (1994), 405–434.

\bibitem{SSVZ} K. Sackel, A. Song, U. Varolgunes, and J. Zhu J. {\it On certain
quantifications of Gromov’s non-squeezing theorem. With an appendix by Jo\'e Brendel,} to appear in Geometry and Topology. arXiv:2105.00586. 

\bibitem{schl} F. Schlenk, Embedding problems in symplectic geometry, 
De Gruyter Expositions in Mathematics 40. Walter de Gruyter Verlag, Berlin, 2005.



\end{thebibliography}
\end{document}